\documentclass[12pt]{article}
\usepackage{hyperref}
\usepackage{cmap}                        
\usepackage[cp1251]{inputenc}            
\usepackage[english]{babel}
\usepackage[left=1in,right=1in,top=1in,bottom=1.5in]{geometry} 
\usepackage{amssymb,amsmath, amsthm, amscd,ifthen}
\usepackage{graphicx}

\newtheorem*{thm*}{Theorem}

\newtheorem{thm}{Theorem}

\newtheorem{lem}[thm]{Lemma}
\newtheorem{cla}[thm]{Claim}

\date{}

\title{Uniform $s$-cross-intersecting families}
\author{Peter Frankl, Andrey Kupavskii\footnote{Moscow Institute of Physics and Technology, \'Ecole Polytechnique F\'ed\'erale de Lausanne; Email: {\tt kupavskii@yandex.ru} \ \ Research supported in part by the Swiss National Science Foundation Grants 200021-137574 and 200020-14453 and by the grant N 15-01-03530 of the Russian Foundation for Basic Research.}}

\date{}
\begin{document}
\maketitle
\begin{abstract} In this paper we study a question related to the classical Erd\H os--Ko--Rado theorem, which states that any family of $k$-element subsets of the set $[n] = \{1,\ldots,n\}$ in which any two sets intersect, has cardinality at most ${n-1\choose k-1}$.

We say that two non-empty families are $\mathcal A, \mathcal B\subset {[n]\choose k}$ are \textit{$s$-cross-intersecting}, if for any $A\in\mathcal A,B\in \mathcal B$ we have $|A\cap B|\ge s$.  In this paper we determine the maximum of $|\mathcal A|+|\mathcal B|$ for all $n$. This generalizes a result of Hilton and Milner, who determined the maximum of $|\mathcal A|+|\mathcal B|$ for nonempty $1$-cross-intersecting families.
\end{abstract}
MSc classification: 05D05
\section{Introduction}
Let $[n]= \{1,\ldots, n\}$ be an $n$-element set and for $0\le k\le n$ let ${[n]\choose k}$ denote the collection of all its $k$-element subsets. Let further $2^{[n]}$ denote the power set. Any subset $\mathcal F\subset 2^{[n]}$ is called a \textit{set family} or \textit{family} for short. The following classical result is one of the cornerstones of extremal set theory. If $F\cap F'\ne \emptyset$ for all $F,F'\in\mathcal F$, then $\mathcal F$ is called \textit{intersecting}.

\begin{thm}[Erd\H os, Ko, and Rado \cite{EKR}] If $\mathcal A\subset{[n]\choose k}$ is intersecting and $n\ge 2k$, then $|\mathcal A|\le {n-1\choose k-1}$ holds.
\end{thm}
This is one of the first results in extremal set theory and probably the first result about intersecting families.

Numerous results extended the Erd\H os--Ko--Rado theorem in different ways. One of the directions was to study \textit{non-trivial} intersecting families, that is, excluding the obvious examples of intersecting families of sets that all contain a fixed element. Note that the family of all sets containing a single element has the size matching the bound from the EKR theorem. Probably, the most known result in this direction is the Hilton--Milner theorem \cite{HM}, which gives the maximum size of a non-trivial intersecting family.

In the same paper Hilton and Milner dealt with pairs of \textit{cross-intersecting} families. We call $\mathcal A,\mathcal B\subset 2^{[n]}$ \textit{cross-intersecting}, if for every $A\in\mathcal A$ and $B\in \mathcal B$ we have $A\cap B\ne \emptyset$. They proved the following inequality:
\begin{thm}\label{thmhm}[Hilton and Milner \cite{HM}] Let $\mathcal A,\mathcal B\subset{[n]\choose k}$ be non-empty cross-intersecting families with $n\ge 2k$. Then $|\mathcal A|+|\mathcal B|\le {n\choose k}-{n-k\choose k}+1.$
\end{thm}

This inequality was generalized by Frankl and Tokushige \cite{FT} to the case $\mathcal A\subset{[n]\choose a}$, $\mathcal B\subset{[n]\choose b}$ with $a\ne b$, and with more general constraints on the sizes of $\mathcal A, \mathcal B$. A simple proof of the theorem above may be found in \cite{FK2}.

We say that two families $\mathcal A,\mathcal B\subset 2^{[n]}$ are \textit{$s$-cross-intersecting}, if for any $A\in\mathcal A, B\in \mathcal B$ we have $|A\cap B|\ge s$.
In this paper we prove the following generalization of the Hilton--Milner theorem. Define the family $\mathcal C = \{C\subset{[n]\choose k}: |C\cap [k]|\ge s\}$.

\begin{thm}\label{thm5} Let $k>s\ge 1$ be integers. Suppose that $\mathcal A,\mathcal B\subset{[n]\choose k}$ are non-empty $s$-cross-intersecting families. Then for $n> 2k-s$ we have
\begin{equation}\label{eqthm5} \max_{\mathcal A,\mathcal B}|\mathcal A|+|\mathcal B|= |\mathcal C|+1.\end{equation}
\end{thm}

\section{Preliminaries}
We start with the following auxiliary statement, which is of independent interest.   Consider a bipartite graph $H$ with parts $V_1,V_2$ and a group of automorphisms $\Gamma$,
where $\Gamma$ is acting on $V_1$ and $V_2$ but respects the parts (that is,
$\forall \gamma \in \Gamma$ and $v\in V_i$ $\gamma(v)\in V_i$ holds).
Let $T_1,\ldots, T_p$ and $S_1,\ldots, S_q$ be the orbits of the action of $\Gamma$, with  $V_1 = T_1\sqcup \ldots\sqcup T_p$ and $V_2 =  S_1\sqcup\ldots \sqcup S_q$.

\begin{lem}\label{lemorbit} There exists an independent set $I$ in $H$ of maximal cardinality, such that for any $i$ either $I\supset T_i$ or $I\cap T_i =\emptyset$ and for any $j$ either $I\supset S_j$ or $I\cap S_j=\emptyset$.
\end{lem}

Make an auxiliary bipartite graph $W$ on the set of vertices $\{T_1,\ldots, T_p\}$ and $\{S_1,\ldots, S_q\}$ and with edges between $T_i$ and $S_j$ iff there is at least one edge connecting a vertex of $T_i$ to a vertex of $S_j$. Put weights $|T_i|, |S_j|$ on the vertices $T_i,S_j$, correspondingly.

Let $B\subset V_1\cup V_2$ be an independent set in $H$ and define $\beta_i = |B\cap T_i|/|T_i|,$ $\beta_{p+j} = |B\cap S_j|/|S_j|$ for $i = 1,\ldots, p, j = 1,\ldots, q$.
\begin{cla} If there is an edge between $T_i$ and $S_j$ in $W$ then
\begin{equation}\label{eqorbit} \beta_i+\beta_{p+j}\le 1.\end{equation}
\end{cla}

\begin{proof} Just note that $\Gamma$ is transitive on both $T_i$ and $S_j$. Therefore, $T_i,S_j$ induce a \textit{biregular} bipartite graph (i.e., degree is constant on each side) with nonzero degrees. The inequality (\ref{eqorbit}) follows easily from the aforementioned regularity, since among the neighbors of a  $\beta_i$-fraction of the vertices from $T_i$ there is at least a $\beta_i$-fraction of the vertices from $S_j$.
\end{proof}
We call any nonnegative vector $(\beta_1,\ldots, \beta_{p+q})$ satisfying (\ref{eqorbit}) a \textit{fractional independent set}. Note that each vector corresponding to an independent set in $W$ is a fractional independent set with coordinates from $\{0,1\}$.

\begin{proof}[Proof of Lemma \ref{lemorbit}]
We take a fractional independent set  $\mathbf v := (\beta_1,\ldots, \beta_{p+q})$ in $W$ with maximal weight $\sum_{i=1}^p \beta_i |T_i|+\sum_{j = 1}^q \beta_{p+j}|S_j|$. The vector  $\mathbf u := (1-\beta_1,\ldots, 1-\beta_{p+q})$ is a minimal weight fractional vertex cover (that is, for each each edge $(T_i,S_j)\in W$ we have $\mathbf u_i+\mathbf u_{p+j}\ge 1$). It is  a standard result in combinatorial optimization that there exists a minimal weight fractional vertex cover $\mathbf u$ in $W$ with integral coordinates (see, e.g., \cite{LRS}). Thus, there exists a maximal weight fractional independent set $\mathbf v$ in $W$ with all coordinates integral. This fractional independent set corresponds to the desired independent set in the graph $H$.
\end{proof}
\vskip+0.2cm

We recall the definition of the \textit{left shifting} (left compression), which we would simply refer to as \textit{shifting}. For a given  pair of indices $i<j\in [n]$ and a set $A \in 2^{[n]}$ we define the $(i,j)$-shift $S_{i,j}(A)$ of $A$ in the following way. If $i\in A$ or $j\notin A$, then $S_{i,j}(A) = A$. If $j\in A, i\notin A$, then $S_{i,j}(A) := (A-\{j\})\cup \{i\}$ that is, $S_{i,j}(A)$ is obtained from $A$  by replacing element $j$ with element $i$.

Next, we define the $(i,j)$-shift $S_{i,j}(\mathcal F)$ of a family $\mathcal F\subset 2^{[n]}$:

$$S_{i,j}(\mathcal F) := \{S_{i,j}(A): A\in \mathcal F\}\cup \{A: A,S_{i,j}(A)\in \mathcal F\}.$$

We call a family $\mathcal F$ \textit{shifted}, if $S_{i,j}(\mathcal F) = \mathcal F$ for all $1\le i<j\le n$.

\section{Proof of Theorem \ref{thm5}}
We assume that $s\ge 2$, as the case $s=1$ is covered by Theorem \ref{thmhm}. Recall that $\mathcal C = \{C\subset{[n]\choose k}: |C\cap [k]|\ge s\}$. It is not hard to see that if $\mathcal A, \mathcal B$ are s-cross-intersecting, then so are
$S_{i,j}(\mathcal A), S_{i,j}(\mathcal B)$. Therefore, we may w.l.o.g. assume that $\mathcal A, \mathcal B$ are shifted and thus both contain $[k]$. It is obvious that in any case $|\mathcal A|+|\mathcal B|\le 2|\mathcal C|$, since $\mathcal A,\mathcal B\subset \mathcal C$. Consider a bipartite graph $G = (\mathcal C-\{[k]\},\mathcal C-\{[k]\}, E)$ with two copies $\mathcal C^1,\mathcal C^2$ of $\mathcal C-\{[k]\}$ as parts and with two sets $C_1,C_2\in \mathcal C$ from different parts being connected by an edge if and only if $|C_1\cap C_2|<s$. Our goal is to show that the maximum independent set in $G$ has size $|\mathcal C|-1$. Indeed, it is clear that $|\mathcal A-\{[k]\}|+|\mathcal B-\{[k]\}|\le \alpha(G)$ and in that case we get the desired bound $|\mathcal A|+|\mathcal B|\le |\mathcal C|+1$. \\


One way to show that $G$ does not have independent sets larger than its parts is to exhibit a perfect matching in $G$. We were unable to do it in general. We are going to circumvent this problem by doing something different. We cover the vertices of a certain weighted graph directly related to $G$ by disjoint edges and paths of even vertex length in such a way that the \textit{total weight} of any independent set in any of these graphs is at most half of the total weight of the set of vertices of that subgraph.\\

We consider the following group of isomorphisms $\Gamma$, acting on $[n]$ (and on the vertex set of $G$). $\Gamma$ is a product of two groups, one consists of all permutations of $[k]$ and the other one consists of all permutations of $[k+1,n]$. Each of the parts of $G$ splits into orbits $\mathcal C_i^j$, where $j= 1,2$ and corresponds to the index of the part, and  $i = s,\ldots, k-1$ and indicates the size of the intersection of sets from $\mathcal C_i^j$ with $[k]$: $\mathcal C_i^j = \{C\in{[n]\choose k}: |C\cap [k]|= i\}$. We refer to these orbits as $\mathcal C_i$ if we think of them as set families.

Take an independent set $B$ in $G$ of largest size. Using Lemma \ref{lemorbit}, we may w.l.o.g. assume that for each $j, i$ either $B\supset \mathcal C^j_i$ or $B\cap \mathcal C^j_i = \emptyset.$ Consider an auxiliary \textit{weighted} graph $W$ as in the proof of Lemma \ref{lemorbit}. That is, the vertices of $W$ are $\mathcal C^j_i$ and the vertices from different parts are connected iff there are two sets from the corresponding orbits that intersect in less than $s$ elements. We also put weights on each vertex of $W$ equal to the cardinality of the corresponding orbit. Thus, in view of Lemma \ref{lemorbit}, it is enough for us to show that one part of $W$ has the same weight as the heaviest independent set in $W$.

Put $n = 2k-s+1+l$, where $l\ge 0$. It is easy to check that $\mathcal C_i^1$ is connected to $\mathcal C^2_t$ with $t \in \{k-i-l,\ldots, k-i+s-1\}\cap \{s,\ldots,k-1\}$. Next we study the weights of the interconnected vertices. We call the value of $i$ \textit{meaningful} if the resulting indices of all families depending on $i$ lie in the set $\{s,\ldots, k-1\}$. We need the following technical lemma.

\begin{lem}\label{lemcoeff} 1. $|\mathcal C_i|\ge |\mathcal C_{k-i+s-1}|$ iff $s\le i\le (k+s-1)/2$. \\
2. $|\mathcal C_{\lfloor\frac {k-l}2\rfloor-i}|\le |\mathcal C_{\lfloor\frac {k+s-1}2\rfloor+i}|$ for all meaningful positive integer values of $i$.
\end{lem}
\begin{proof} The proof of the lemma is just a careful algebraic manipulation. We remark that for the first reading one may omit all the integer parts in the computations below to make the verification easier.

1. Assume $i\le (k+s-1)/2$. We have $|\mathcal C_i| = {k\choose i}{n-k\choose k-i}$, $|\mathcal C_{k-i+s-1}| = {k\choose k-i+s-1}{n-k\choose i-s+1}.$ Therefore,

$$\frac{|\mathcal C_i|}{|\mathcal C_{k-i+s-1}|} = \frac{(i-s+1)!^2(k-i+s-1)!(n-k-i+s-1)!}{(k-i)!^2i!(n-2k+i)!} = $$$$ \frac{(k-i+s-1)\cdots(k-i+1)}{i\cdots(i-s+2)}\cdot\frac{(n-k-i+s-1)\cdots(n-2k+i+1)}{(k-i)\cdots(i-s+2)}\ge 1,$$

because $k-i+s-1\ge i$ since $i\le (k+s-1)/2,$ and $n-k-i+s-1\ge k-i$ since $n\ge 2k-s+1$.\\

2. Assume first that $\lfloor\frac{k-l}2\rfloor \ge \lceil\frac{k-s+1}2\rceil$. We have

$$\frac{|\mathcal C_{\lfloor\frac {k+s-1}2\rfloor+i}|}{|\mathcal C_{\lfloor\frac {k-l}2\rfloor-i}|} = \frac{(\lceil\frac{k+l}2\rceil+i)!^2(\lfloor\frac{k-l}2\rfloor-i)!(n-k-\lceil\frac{k+l}2\rceil-i)!}{(\lceil\frac{k-s+1}2\rceil-i)!^2 (\lfloor\frac{k+s-1}2\rfloor+i)!(n-k-\lceil\frac{k-s+1}2\rceil+i)!} = $$
$$ \frac{(\lfloor\frac {k-l}2\rfloor-i) \cdots(\lceil\frac{k-s+1}2\rceil-i+1)}{(\lfloor\frac{k+s-1}2\rfloor+i)\cdots(\lceil\frac{k+l}2\rceil+i+1)} \cdot \frac{(\lceil\frac{k+l}2\rceil+i)\cdots(\lceil\frac{k-s+1}2\rceil-i+1)}{(n-k-\lceil\frac{k-s+1}2\rceil+i)\cdots(n-k-\lceil\frac{k+l}2\rceil-i+1)} =: (*).$$
We note that $n = 2k-s+1+l$, and therefore $n-k-\lceil\frac{k-s+1}2\rceil+i = \lfloor\frac{k-s+1}2\rfloor+l+i \le (\lceil\frac{k+l}2\rceil+i)-\lceil\frac{s-1-l}2\rceil.$

\begin{small}$$(*) \ge  \frac{(\lfloor\frac {k-l}2\rfloor-i) \cdots(\lceil\frac{k-s+1}2\rceil-i+1)}{(\lfloor\frac{k+s-1}2\rfloor+i)\cdots(\lceil\frac{k+l}2\rceil+i+1)} \cdot
\frac{(\lceil\frac{k+l}2\rceil+i)\cdots(\lceil\frac{k+l}2\rceil+i-\lceil\frac{s-1-l}2\rceil+1)} {(n-k-\lceil\frac{k+l}2\rceil+\lceil\frac{s-1-l}2\rceil-i)\cdots(n-k-\lceil\frac{k+l}2\rceil-i+1)}
$$\end{small}
In the first fraction the number of factors in both denominator and numerator is the same and is $\lfloor\frac {k-l}2\rfloor- \lceil\frac{k-s+1}2\rceil$, which is at most $\lceil\frac{s-1-l}2\rceil$. The ratio of the $j+1$-st factors in the first fraction, that is, the expression $\frac{\lfloor\frac {k-l}2\rfloor-i-j}{\lfloor\frac{k+s-1}2\rfloor+i-j}$ for $j\in \{0,\ldots, \lceil\frac{s-1-l}2\rceil-1\}$ is at most $1$ and thus does not increase as $j$ increases.

In the second fraction the number of factors in both the numerator and denominator is $\lceil\frac{s-1-l}2\rceil$. Similarly to the first fraction, the expression $\frac{\lceil\frac{k+l}2\rceil+i-j}{n-k-\lceil\frac{k+l}2\rceil+\lceil\frac{s-1-l}2\rceil-i-j}$ for $j\in \{0,\ldots, \lceil\frac{s-1-l}2\rceil-1\}$ is at least $1$ and thus does not decrease as $j$ increases. Therefore, the product of these two fractions is at least

$$\left(\frac{\lceil\frac{k-s+1}2\rceil-i}{\lceil\frac{k+l}2\rceil+i}\right)^{\lceil\frac{s-1-l}2\rceil} \left(\frac{\lceil\frac{k+l}2\rceil+i}{n-k-\lceil\frac{k+l}2\rceil+\lceil\frac{s-1-l}2\rceil-i}\right)^{\lceil\frac{s-1-l}2\rceil} =$$
$$\left(\frac{\lceil\frac{k-s+1}2\rceil-i}{k-s+l+1-\lceil\frac{k+l}2\rceil+\lceil\frac{s-1-l}2\rceil-i}\right)^{\lceil\frac{s-1-l}2\rceil} \ge 1.$$

The last inequality holds since without integer parts the denominator is equal to the enumerator, and the possible gain because of the integer parts in the denominator is at most $1/2$. However, both numbers are integer, thus the denominator is not bigger than the enumerator.\\

Assume now that $\lfloor\frac{k-l}2\rfloor < \lceil\frac{k-s+1}2\rceil$.

$$\frac{|\mathcal C_{\lfloor\frac {k+s-1}2\rfloor+i}|}{|\mathcal C_{\lfloor\frac {k-l}2\rfloor-i}|} = \frac{(\lceil\frac{k+l}2\rceil+i)!^2(\lfloor\frac{k-l}2\rfloor-i)!(n-k-\lceil\frac{k+l}2\rceil-i)!}{(\lceil\frac{k-s+1}2\rceil-i)!^2 (\lfloor\frac{k+s-1}2\rfloor+i)!(n-k-\lceil\frac{k-s+1}2\rceil+i)!} = $$
$$ \frac{(\lceil\frac{k+l}2\rceil+i)\cdots(\lfloor\frac{k+s-1}2\rfloor+i+1)}{ (\lceil\frac{k-s+1}2\rceil-i)\cdots(\lfloor\frac {k-l}2\rfloor-i+1)} \cdot \frac{(\lceil\frac{k+l}2\rceil+i)\cdots(\lceil\frac{k-s+1}2\rceil-i+1)}{(n-k-\lceil\frac{k-s+1}2\rceil+i)\cdots(n-k-\lceil\frac{k+l}2\rceil-i+1)} =: (*).$$

We note that $n = 2k-s+1+l$, and therefore $n-k-\lceil\frac{k-s+1}2\rceil+i = \lfloor\frac{k-s+1}2\rfloor+l+i \le (\lceil\frac{k+l}2\rceil+i)+\lfloor\frac{l-s+1}2\rfloor.$

\begin{multline*}(*) \ge \frac{(\lceil\frac{k+l}2\rceil+i)\cdots(\lfloor\frac{k+s-1}2\rfloor+i+1)}{ (\lceil\frac{k-s+1}2\rceil-i)\cdots(\lfloor\frac {k-l}2\rfloor-i+1)}
\cdot\\
\frac{(\lceil\frac{k-s+1}2\rceil+\lfloor\frac{l-s+1}2\rfloor-i)\cdots(\lceil\frac{k-s+1}2\rceil-i+1)} {(n-k-\lceil\frac{k-s+1}2\rceil+i)\cdots(n-k-\lceil\frac{k-s+1}2\rceil- \lfloor\frac{l-s+1}2\rfloor+i+1)}.\end{multline*}
Bounding the two fractions similarly to how it was done in the previous case, one can see that this expression is at least
$$\left(\frac{\lceil\frac{k+l}2\rceil+i+1}{\lceil\frac{k-s+1}2\rceil-i+1}\right)^{\lfloor\frac{l-s+1}2\rfloor} \left(\frac{\lceil\frac{k-s+1}2\rceil-i+1}{n-k-\lceil\frac{k-s+1}2\rceil- \lfloor\frac{l-s+1}2\rfloor+i+1}\right)^{\lfloor\frac{l-s+1}2\rfloor} =$$
$$\left(\frac{\lceil\frac{k+l}2\rceil+i+1}{k-s+l+1-\lceil\frac{k-s+1}2\rceil- \lfloor\frac{l-s+1}2\rfloor+i+1}\right)^{\lfloor\frac{l-s+1}2\rfloor} \ge 1.$$
As in the previous case, in the last fraction the denominator is equal to the enumerator if one removes all integer parts, and as in the previous case we conclude that the last inequality is valid.
\end{proof}
\vskip+0.2cm

The next crucial step is to decompose the graph $W$ into symmetric chains of even length. Let  $(j,\bar j) = (1,2)$ or $(2,1)$. We know that $\mathcal C^{j}_i$ is connected to $\mathcal C^{\bar j}_{k+s-1-i}$ for all $i$.  We call these edges \textit{edges of the first type}. Next, $\mathcal C^j_{i}$ is connected to $\mathcal C^{\bar j}_i$ if $ \lceil\frac{k-l}2\rceil\le i< \frac{k+s-1}2.$ We call these edges \textit{edges of the second type}. Finally, $\mathcal C^j_{\lfloor\frac {k-l}2\rfloor-i}$ is connected to $\mathcal C^{\bar j}_{\lfloor\frac {k+s-1}2\rfloor+i}$ for each meaningful $i\ge 1$ and $s\ge 2$. We call these edges \textit{edges of the third type}.

We claim that the graph $W$ is decomposed into paths of even length using the three types of edges above. See Fig. 1, where the two parts are represented by two horizontal lines of points, and the values near the points indicate the size of the intersection of the sets of the corresponding family with $[k]$. The arrows on the edges indicate the direction from the vertex of smaller weight to the vertex of larger weight. The black two-directional edges are the edges of the second type and connect the vertices of the same weight. The vertical edges are the edges of the first type, and the other edges are the edges of the third type.

The middle vertical line corresponds to the size of the intersection with $[k]$ equal to $\frac{k+s-1}2$ in both parts. The two dashed lines correspond to the intersection size $\lceil \frac{k-l}2\rceil$ for one of the parts. Between these two lines each vertex has an edge of the second type. Note that the starting vertex of the edges of the third type is always outside the region bounded between the two dashed lines. The edges of the third type may or may not intersect the dashed lines, but they never intersect the black line.

\begin{center}  \includegraphics[width=150mm]{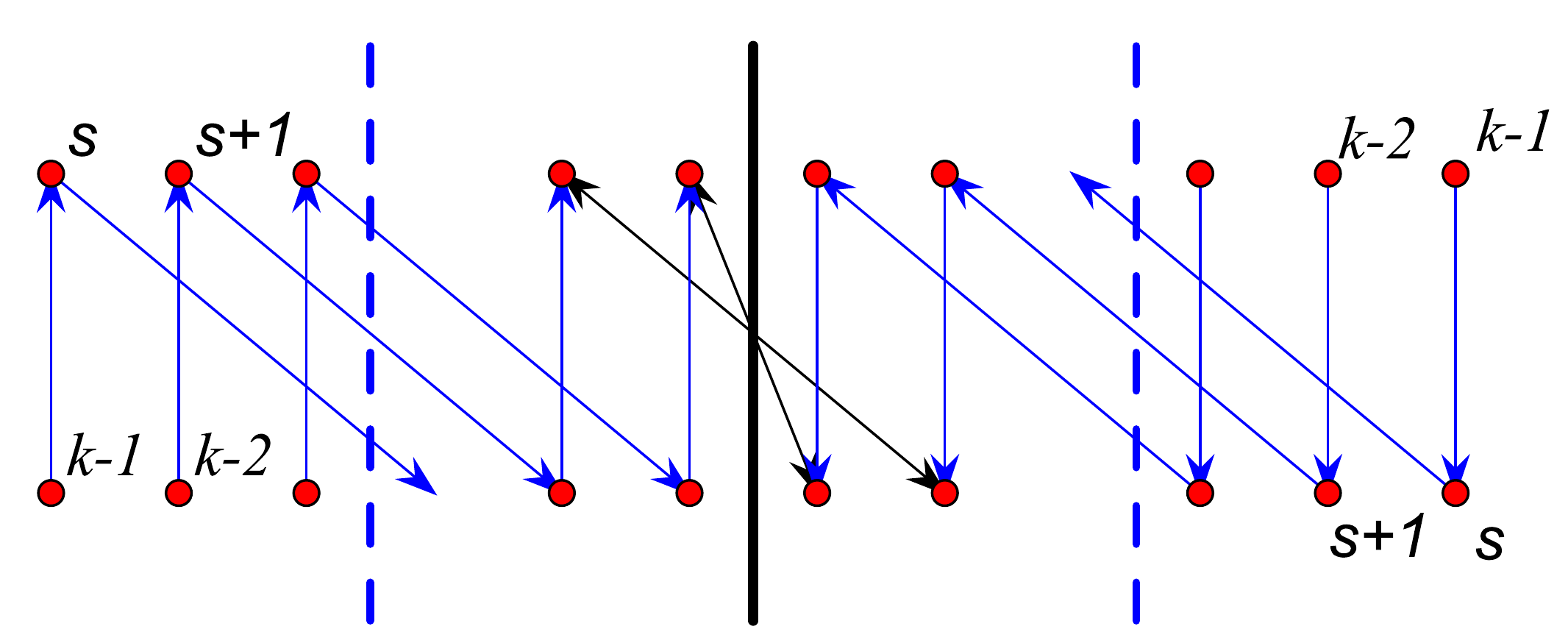}\label{pic1}  \end{center}\begin{center}  Figure 1. Decomposition of the graph $W$ into symmetric chains. \end{center}

In the middle of each path $P$ there is an edge of the second type, connecting two vertices of the same weight. Due to Lemma \ref{lemcoeff}, for each path the weight of the vertex does not decrease as we move along the path towards the middle edge of the second type. Therefore, it is easy to see that any independent set in $P$ has weight less than or equal to the half of the weight of $P$. This means that any independent set in $W$ cannot have weight bigger than the half of the weight of $W$. This concludes the proof of Theorem \ref{thm5}.

\subsubsection*{Acknowledgements} We thank the anonymous referee for carefully reading the paper and giving several suggestions that helped to improve the presentation of the results.


\begin{thebibliography}{}

\bibitem{EKR} P. Erd\H os, C. Ko, R. Rado, \textit{Intersection theorems for systems of finite sets}, The Quarterly Journal of Mathematics, 12 (1961) N1, 313--320.

\bibitem{FK2} P. Frankl, A. Kupavskii, \textit{A size-sensitive inequality for cross-intersecting families}, submitted.

\bibitem{FT} P. Frankl, N. Tokushige, \textit{Some best possible inequalities concerning cross-intersecting families}, Journal of Combinatorial Theory, Ser. A 61 (1992), N1, 87--97.


\bibitem{HM} A.J.W. Hilton, E.C. Milner, \textit{Some intersection theorems for systems of finite sets}, Quart. J. Math. Oxford 18 (1967), 369--384.

\bibitem{LRS} L.C. Lau, R. Ravi, M. Singh \textit{Iterative methods in combinatorial optimization} (2011), Cambridge University Press.
\end{thebibliography}
\end{document}